\documentclass[11pt]{article}
\usepackage{amsthm}
\usepackage{amsmath}
\usepackage{amssymb}
\usepackage{latexsym}
\usepackage{amsfonts}
\usepackage{placeins}
\usepackage{caption}
\usepackage{mathrsfs}
\usepackage[latin1]{inputenc}
\usepackage{listings}
\newtheorem{theorem}{\bf Theorem}[section]

\newtheorem{lemma}[theorem]{\bf Lemma}
\begin{document}
\title{$k$-Pell-Lucas numbers which are concatenations of two repdigits}
\author{Bibhu Prasad Tripathy and Bijan Kumar Patel}
\date{}
\maketitle
\begin{abstract}
For any integer $k \geq 2$, let $\{Q_{n}^{(k)} \}_{n \geq -(k-2)}$ denote the $k$-generalized Pell-Lucas sequence which starts with $0, \dots ,2,2$($k$ terms) where each next term is the sum of the $k$ preceding terms. In this paper, we find all the $k$-generalized Pell-Lucas numbers that are concatenations of two repdigits.
\end{abstract} 


\section{Introduction}
The Pell-Lucas sequence $\{Q_n\}_{n\geq 0}$ is the binary recurrence sequence defined by 
\[
Q_{n+2} = 2 Q_{n+1} + Q_{n}~~ {\rm for} ~n\geq 0
\]
with the initial terms $Q_{0} = 2$ and $Q_{1} = 2$. 

Let $ k \geq 2 $ be an integer. We consider the $k$-generalized Pell-Lucas sequence $\{Q_{n}^{(k)} \}_{n \geq -(k-2)}$, which is a generalization of the Pell-Lucas sequence and is given by the recurrence
\[
Q_{n}^{(k)} = 2 Q_{n-1}^{(k)} + Q_{n-2}^{(k)} + \dots + Q_{n-k}^{(k)} ~~\text{for all}~n \geq 2,  
\]
with the initial terms $Q_{-(k-2)}^{(k)} = Q_{-(k-3)}^{(k)} = \dots = Q_{-1}^{(k)} = 0,  Q_{0}^{(k)} = 2$ and $Q_{1}^{(k)} = 2$. The expression $Q_{n}^{(k)}$ denotes the $n^{th}$ term of the $k$-Pell-Lucas sequence. The usual Pell-Lucas sequence $\{Q_n\}_{n\geq 0}$ is obtained when $k = 2$.


A natural number is called base $b$ repdigit if all its base $b$-digits are equal. Let $b > 1$ be any positive integer. A positive integer $N$ is called $b$ repdigit if it is of  the form 
\[
N = a\left( \frac{b^{l} - 1}{b-1} \right) = {\overline{
\mathop{\underbrace{a \cdots a}\limits_\text{$m$ times}
}}}\vphantom{0}_{~(b)},
\]
for some positive integers $a, l$ with $a \in \{1,2, \cdots, b-1 \}$ and $l \geq 1$. When $b = 10$, then it is said to be repdigit. For $k \geq 1$ and $b > 1$, a positive integer of the form 
\[
N = {\overline{
\mathop{\underbrace{a_{1} \cdots a_{1}}}\limits_\text{$m_{1}$ times}
\mathop{\underbrace{a_{2} \cdots a_{2}}}\limits_\text{$m_{2}$ times} \cdots
\mathop{\underbrace{{a_{k} \cdots a_{k}}}}\limits_\text{$m_{k}$ times}}}\vphantom{0}_{~(b)}
\]
where $a_{1}, a_{2}, \cdots , a_{k} \in  \{0, 1, 2, \cdots, b-1 \}$ with $a_{1} > 0$, can be called as concatenation of $k$ repdigits in base $b$.

In recent years, the concatenation of repdigits in different linear recurrence sequences has been studied by many researchers.  For example, Alahmadi et al. \cite{Alahmadi} studied the problem of finding all Fibonacci numbers which are concatenations of two repdigits. Following that Rayguru and Panda \cite{Rayguru} found the balancing number which is the concatenations of two repdigits. Erduvan and Keskin \cite{Erduvan} studied the Lucas numbers which are concatenations of two repdigits. Consequently, Ddamulira searched all  Tribonacci and Padovan numbers which are concatenations of two repdigits in \cite{Ddamulira} and \cite{Ddamulira1} respectively. Batte et al. \cite{Batte} searched for only Perrin numbers which are concatenations of two repdigits. 

Recently, Alahmadi \cite{Alahmadi1} found $k$-generalized Fibonacci numbers that are concatenations of two repdigits that have at most four digits. Consequently, Bravo et al. \cite{Bravo2} searched for all $k$-Lucas numbers that are concatenations of two repdigits. Later, \c{S}iar and Keskin \cite{Siar2} showed that 12, 13, 29, 33, 34, 70, 84, 88, 89, 228 and  $233$ are the only $k$-generalized Pell numbers, which are concatenations of two repdigits with at least two digits. In this paper, we find all the $k$-Pell-Lucas numbers that are concatenations of two repdigits. More precisely, we solve the Diophantine equation
\begin{equation}\label{eq 1.2}
  Q_{n}^{(k)} = a \left ( \frac{10^{l}-1}{9} \right) \cdot 10^{m} + b \left ( \frac{10^{m}-1}{9} \right),
\end{equation}
in non-negative integers $a, b, k, l, m, n$ with $a, b \in \{0, 1,\dots, 9\}$, $a > 0$, $k \geq 2$ and $l, m, n \geq 1$. In particular, our main result is the following.
\begin{theorem}\label{thm1}
The only $k$-Pell-Lucas numbers that are concatenations of two repdigits are
\begin{align*}
& Q_{1}^{(k)} = 2 \ \text{for} \ k \geq 2; & Q_{2}^{(k)} = 6 \  \text{for} \ k \geq 2; \\
&  Q_{3}^{(k)} = 16 \ \text{for} \  k \geq 3; \ & Q_{4}^{(k)} = 42 \ \text{for} \ k \geq 4; \\
& Q_{5}^{(k)} = 110 \ \text{for} \ k \geq 5; & Q_{6}^{(k)} = 2288 \ \text{for} \  k \geq 6; \\
& Q_{3}^{(2)} = 14,\quad Q_{4}^{(2)} = 34; \quad & Q_{4}^{(3)} = 40, \quad Q_{5}^{(2)} = 82;  \\
& Q_{7}^{(3)} = 662. \\
\end{align*}
\end{theorem} 
For the proof of Theorem \ref{thm1}, we first find an upper bound for $n$ in terms of $k$ by applying Matveev's result on linear forms in logarithms \cite{Matveev}. When $k$ is small, the theory of continued fractions suffices to lower such bounds and complete the calculations. When $k$ is large, we use the fact that the dominant root of the $k$-Pell-Lucas sequence is exponentially close to $\phi^2$ {\rm{(see \cite{Siar})}} where $\phi$ denotes the golden section. So we use this estimation in our further calculation with linear forms in logarithms to obtain absolute upper bounds for $n$ which can be reduced by using  Dujella and Peth\"{o}'s result \cite{Dujella}.

\section{Preliminary Results}
\subsection{Properties of $k$-generalized Pell-Lucas sequence}
The characteristic polynomial of the $k$-generalized Pell-Lucas sequence is  
 \[
 \Phi_{k}(x) = x^{k} - 2 x^{k-1}  - x^{k-2} - \dots - x -1. 
 \]
The above polynomial is irreducible over $\mathbb{Q} [x]$ and it has one positive real root $\gamma := \gamma(k)$ which is located between $ \phi^ 2(1 - \phi ^{-k})$ and $\phi^2$, lies outside the unit circle (see \cite{Siar}). The other roots are firmly contained within the unit circle. To simplify the notation, we will omit the dependence on $k$ of $\gamma$ whenever no confusion may arise.

The Binet  formula for $Q_{n}^{(k)}$  that found in \cite{Siar} is
\begin{equation}\label{eq 2.3}
Q_{n}^{(k)} = \displaystyle\sum_{i=1}^{k} (2 \gamma_{i} - 2) g_{k} (\gamma_{i})\gamma_{i}^{n} = \sum_{i=1}^{k} \frac{2(\gamma_{i}-1)^{2}}{(k+1)\gamma_{i}^2 - 3k\gamma_{i} + k -1} \gamma_{i}^{n} ,
\end{equation}
where  $\gamma_{i}$ represents the roots of the characteristic polynomial $\Phi_{k}(x)$ and the function $g_{k}$ is  given by
\begin{equation}\label{eq 2.4}
	g_{k}(z) := \frac{z-1}{(k+1)z^2 - 3kz + k -1}, 
\end{equation}
for an integer $k$ $\geq$ 2.
\\
Additionally, it is  also shown in \cite[Lemma~10]{Siar} that the roots located inside the unit circle have a very minimal influence on the formula \eqref{eq 2.5}, which is given by the approximation 
\begin{equation}\label{eq 2.5}
	\left| Q_{n}^{(k)} -  (2\gamma-2)g_{k} (\gamma) \gamma^{n} \right| < 2 \quad \text{holds  for   all} \quad n \geq 2 - k.
\end{equation}
Furthermore, it is shown by \c{S}iar and Keskin in \cite[Lemma~10]{Siar} that the inequality
\begin{equation}\label{eq 2.6}
	\gamma^{n-1} \leq Q_{n} ^ {(k)} \leq 2 \gamma^{n} \text{ holds for all } n \geq 1 ~{\rm and}~ k \geq 2.
\end{equation}
\begin{lemma}\label{lem 2.1}{\rm{(\cite{Bravo}, Lemma 3.2)}}.
Let $k \geq 2 $  be an integer. Then we have
\[
0.276 < g_{k}(\gamma) < 0.5 \  and \  \left| g_{k}(\gamma_{i}) \right| < 1 \quad for \quad 2 \leq i \leq k.
\]	
\end{lemma}
\noindent
Furthermore, \c{S}iar et al. \cite{Siar1} showed that the logarithmic height of $g_{k} (\gamma)$ satisfies
\begin{equation}\label{eq 2.7}
    h(g_{k}(\gamma)) < 5 \log k \quad \text{for all} \ k \geq 2.
\end{equation}
\begin{lemma}\label{lem 2.2}
Let $\gamma$ be the dominant root of the characteristic polynomial $\Phi_{k}(x)$  and consider the function $g_{k}(x)$ defined in \eqref{eq 2.4}. If $k \geq 50$ and $n > 1$ are integers satisfying $n < \phi^{k/2}$, then the following inequalities holds
\\
(i){\rm{(\cite{Siar}, Equation~30)}}
\[
\left| (2 \gamma - 2)\gamma^{n} - 2 \phi^{2n+1} \right| < \frac{4 \phi^{2n}}{\phi^{k/2}},
\]
(ii){\rm{(\cite{Siar}, Lemma~13)}}
\[
|g_{k}(\gamma) -  g_{k}(\phi^{2})| < \frac{4k}{\phi^{k}}.
\]
\end{lemma}
\begin{lemma}\label{lem 2.3}
Let $k \geq 50$ and suppose that $n < \phi^{k/2}$, then
\begin{equation}\label{eq 2.8}
   (2 \gamma - 2)g_{k} (\gamma) \gamma^{n} = \frac{2 \phi^{2n+1}}{\phi + 2}(1 + \xi), \quad \text{where} \quad  |\xi| < \frac{1.25}{\phi^{k/2}}.
\end{equation}
\end{lemma}
\begin{proof}
By virtue of Lemma \ref{lem 2.2}, we have
\begin{equation*}
(2 \gamma - 2)\gamma^{n} = 2 \phi^{2n+1} + \delta \quad \text{and} \quad g_{k}(\gamma) = g_{k}(\phi^{2}) + \eta, 
\end{equation*}
where
\begin{equation}\label{eq 2.9}
 |\delta| < \frac{4 \phi^{2n}}{\phi^{k/2}} \quad \text{and} \quad |\eta| < \frac{4k}{\phi^{k}}.   
\end{equation}
Since $g_{k}(\phi^{2}) = \frac{1}{\phi + 2}$, we can write
\begin{align}\label{eq 2.10}
(2 \gamma - 2)\gamma^{n}g_{k}(\gamma) & = \left( 2 \phi^{2n+1} + \delta \right)  \left( g_{k}(\phi^{2}) + \eta\right) \nonumber \\ 
&  = \frac{2 \phi^{2n+1}}{\phi+2}(1 + \xi),
\end{align}
where 
\[
\xi = \frac{\delta}{2 \phi^{2n+1}} + (\phi +2)\eta + \frac{(\phi +2)\eta \delta}{2 \phi^{2n+1}}.
\]
Using \eqref{eq 2.9} in \eqref{eq 2.10}, we obtain
\[
|\xi|  < \frac{2/\phi}{\phi^{k/2}} + \frac{4k(\phi+2)}{\phi^{k}} + \frac{8k(\phi+2)}{\phi \cdot \phi^{3k/2}} < \frac{1.25}{\phi^{k/2}},
\]
where we have used the facts 
\[
\frac{2/\phi}{\phi^{k/2}} < \frac{1.24}{\phi^{k/2}}, \quad  \frac{4k(\phi+2)}{\phi^{k}} < \frac{0.005}{\phi^{k/2}} \quad \text{and} \quad \frac{(8k/\phi)(\phi+2)}{\phi^{3k/2}} < \frac{0.005}{\phi^{k/2}}, 
\]
for all  $k \geq 50$. This completes the proof.
\end{proof}

Finally, we give the following estimate which is an important key to our analysis. This is a direct consequence of Lemma \ref{lem 2.3}.
\begin{lemma}\label{lem 2.4}
Let $k \geq 2$ and suppose that $2n \geq k/2$. If $n < \phi^{k/2}$, then
\[
Q_{n}^{(k)} = \frac{2\phi^{2n+1}}{\phi +2} (1 + \zeta ) \quad \text{where} \quad |\zeta| < \frac{41}{\phi^{k/2}}.
\]
\end{lemma}

\subsection{Linear forms in logarithms}
Let $\gamma$ be an algebraic number of degree $d$ with a minimal primitive polynomial 
\[
f(Y):= b_0 Y^d+b_1 Y^{d-1}+ \cdots +b_d = b_0 \prod_{j=1}^{d}(Y- \gamma^{(j)}) \in \mathbb{Z}[Y],
\]
where the $b_j$'s are relatively prime integers, $b_0 >0$, and the $\gamma^{(j)}$'s are conjugates of $\gamma$. Then the \emph{logarithmic height} of $\gamma$ is given by
\[
h(\gamma)=\frac{1}{d}\left(\log b_0+\sum_{j=1}^{d}\log\left(\max\{|\gamma^{(j)}|,1\}\right)\right).
\]
With the above notation, Matveev (see  \cite{Matveev} or  \cite[Theorem~9.4]{Bugeaud}) proved the following result.

\begin{theorem}\label{thm2}
Let $\eta_1, \ldots, \eta_s$ be positive real algebraic numbers in a real algebraic number field $\mathbb{L}$ of degree $d_{\mathbb{L}}$. Let $a_1, \ldots, a_s$ be non-zero  integers such that
\[
\Lambda :=\eta_1^{a_1}\cdots\eta_s^{a_s}-1 \neq 0.
\]
Then
\[
- \log  |\Lambda| \leq 1.4\cdot 30^{s+3}\cdot s^{4.5}\cdot d_{\mathbb{L}}^2(1+\log d_{\mathbb{L}})(1+\log D)\cdot B_1 \cdots B_s,
\]
where
\[
D\geq \max\{|a_1|,\ldots,|a_s|\},
\]
and
\[
B_j\geq \max\{d_{\mathbb{L}}h(\eta_j),|\log \eta_j|, 0.16\}, ~ \text{for all} ~ j=1,\ldots,s.
\]
\end{theorem}

\subsection{Reduction method}
Here we present the following result due to Dujella and Peth\"{o} \cite[Lemma~5 (a)]{Dujella} which is a generalization of a result of Baker and Davenport's result \cite{Baker}.
\begin{lemma}\label{lem 2.6}
Let $\widehat{\tau}$ be an irrational number, and let $A,C,\widehat{\mu}$ be some real numbers with $A>0$ and $C>1$. Assume that $M$ is a positive integer, and let $p/q$ be a convergent of the continued fraction of the irrational $\widehat{\tau}$ such that $q > 6M$. Put \[\epsilon:=||\widehat{\mu} q||-M||\widehat{\tau} q||,
\]
where $||\cdot||$ denotes the distance from the nearest integer.  If $\epsilon >0$, then there is no solution to the inequality 
\[
0< |r \widehat{\tau}-s+\widehat{\mu}| <AC^{-t},
\]
in positive integers $r$, $s$ and $t$ with
\[
r \leq M \quad\text{and}\quad t \geq \frac{\log(Aq/\epsilon)}{\log C}.
\]
\end{lemma}
\subsection{Useful Lemmas}
We conclude this section by recalling two results that we will need in this work.

\begin{lemma}\label{lem 2.7} {\rm{(\cite{Weger}, Lemma 2.2)}}
Let $a, x \in \mathbb{R}$. If $0< a < 1$ and $|x| < a$, then 
\[
|\log(1+x)| < \frac{-\log(1-a)}{a}\cdot |x|
\]
and 
\[
|x| < \frac{a}{1 - e^{-a}} \cdot |e^{x} - 1|.
\]
\end{lemma}
\begin{lemma}\label{lem 2.8} {\rm{(\cite{Sanchez}, Lemma 7)}}
If $m \geq 1$, $S \geq (4m^{2})^{m}$ and $\frac{x}{(\log x)^{m}} < S$, then $x < 2^{m} S (\log S)^{m}$.
\end{lemma}

\section{Proof of Theorem \ref{thm1}}
First, observe that for $1 \leq n \leq k+1$,  we have  $Q_{n}^{(k)} = 2 F_{2n}$ (see $\cite[Lemma~10]{Siar}$) where $F_{m}$ is the $m$th Fibonacci number. Therefore \eqref{eq 1.2} becomes 
\begin{equation*}
  2 F_{2n} = a \left ( \frac{10^{l}-1}{9} \right) \cdot 10^{m} + b \left ( \frac{10^{m}-1}{9} \right).
\end{equation*}
So by using \textit{Mathematica}, we found that the solutions of the above equation satisfy for $n \in \{ 1,2,3,4,5,6 \}$. Hence, in the sequel, we suppose that $n \geq k+2$ and $k \geq 2$.

\subsection{An initial relation between $n$ and $l+m$}
In view of equation \eqref{eq 1.2}, we have
\begin{equation}\label{eq 3.12}
10^{l+m-1} < a \left ( \frac{10^{l}-1}{9} \right) \cdot 10^{m} + b \left (\frac{10^{m}-1}{9} \right) < 10^{l+m}.    
\end{equation}
Indeed, it follows from \eqref{eq 2.6} and \eqref{eq 3.12} that
\[
10^{l+m-1} < a \left ( \frac{10^{l}-1}{9} \right) \cdot 10^{m} + b \left ( \frac{10^{m}-1}{9} \right) =  Q_{n}^{(k)} < 2 \gamma^{n}
\]
and 
\[
\gamma^{n-1} < Q_{n}^{(k)} =  a \left ( \frac{10^{l}-1}{9} \right) \cdot 10^{m} + b \left ( \frac{10^{m}-1}{9} \right) < 10^{l+m} .
\]
Thus, taking the logarithm on both sides of the above inequalities, we obtain
\[
\left(l + m-1\right)\frac{\log 10}{\log \gamma} - \frac{\log 2}{\log \gamma} < n  < \left( l + m\right) \frac{\log 10}{\log \gamma}+ 1.
\]
Furthermore, since $2 < \gamma < \phi^{2} < 3$, it follows from the above that
\begin{equation}\label{eq 3.13}
    2(l + m)-2.6 < n < 3.4(l + m)+1.
\end{equation}
\subsection{Upper bounds for $n$ in terms of $k$}
Now, rearranging the equation \eqref{eq 1.2} as 
\[
9 \cdot  Q_{n}^{(k)}  = a(10^{l} -1)\cdot 10^{m} + b(10^{m}-1).
\]
Using inequality \eqref{eq 2.5} and taking absolute values on both sides, we get
\begin{align} \label{eq 3.14}
 \left| 9 (2 \gamma -2) g_{k}(\gamma) \gamma^{n} - a \cdot 10^{l+m} \right| & = \left| - 10^{m}(a-b) - b - 9 Q_{n}^{(k)}  + 9(2\gamma - 2)g_{k}(\gamma) \gamma^{n} \right| \nonumber \\ 
 & \leq 10^{m}(a-b) + b + 9 \cdot 2 \nonumber \\ 
 &  \leq 9 \cdot 10^{m} + 27  \nonumber \\ 
 & <  11.8 \cdot 10^{m} 
\end{align}
for $m \geq 1$. Dividing both sides of the inequality \eqref{eq 3.14} by $a \cdot 10^{l+m}$, we obtain \begin{equation}\label{eq 3.15}
  \left|10^{-(l+m)} \gamma^{n} \frac{ 9 (2 \gamma -2) g_{k}(\gamma)}{a}  - 1 \right|  < \frac{11.8 \cdot 10^{m}}{a 10^{l+m}} < \frac{11.8}{10^{l}}.
\end{equation}
Let
\begin{equation} \label{eq 3.16}
    \Lambda_{1} := 10^{-(l+m)} \gamma^{n} \frac{ 9 (2 \gamma -2) g_{k}(\gamma)}{a}  - 1.
\end{equation}
If $\Lambda_{1} = 0$, then we would get 
\[
\frac{a \cdot 10^{l+m}}{9} = (2\gamma - 2)g_{k}(\gamma) \gamma^{n} =   \frac{2(\gamma-1)^{2}}{(k+1)\gamma^2 - 3k\gamma + k -1} \gamma^{n}.
\]
Conjugating the above relation by some automorphism of the Galois group of the splitting field of  $ \Phi_{k}(x)$ over $\mathbb{Q}$ and taking absolute values, we get
\[
\frac{a \cdot 10^{l+m}}{9} = (2\gamma - 2)g_{k}(\gamma) \gamma^{n} =  \left| \frac{2(\gamma_{i}-1)^{2}}{(k+1)\gamma_{i}^2 - 3k\gamma_{i} + k -1} \gamma_{i}^{n} \right|
\]
for some $i > 1$, where $\gamma_{i}$  represents the roots of the characteristic polynomial $\Phi_{k}(x)$. Using the facts $|\gamma_{i}| < 1$ and $|g_{k}(\gamma_{i})|<1$, we can write 
\[
\frac{a \cdot 10^{l+m}}{9} = 2| \gamma_{i}-1| \left| \frac{(\gamma_{i}-1)}{(k+1)\gamma_{i}^2 - 3k\gamma_{i} + k -1} \right| |\gamma_{i}^{n}| < 4,
\]
which is impossible as $l + m \geq 2$. Hence $\Lambda_{1} \neq 0$. To apply Theorem \ref{thm2} to $\Lambda_{1}$ given by \eqref{eq 3.16}, we set:
\[
\eta_{1} := 10 , \quad \eta_{2} := \gamma , \quad \eta_{3} := \frac{9(2 \gamma-2)g_{k}(\gamma)}{a},
\]
and
\[ a_{1}:= -(l+m), \quad a_{2}:= n, \quad a_{3}:= 1.
\]
Note that the algebraic numbers $\eta_{1}, \eta_{2}, \eta_{3}$ belongs to the field  $\mathbb{L} := \mathbb{Q}(\gamma)$ and $d_{\mathbb{L}} = [\mathbb{L}:\mathbb{Q}] \leq k$. Since $h(\eta_{1}) = \log 10 $ and  $h(\eta_{2}) = (\log \gamma) / k < ( \log 3)/k$, it follows that 
\[
\max\{kh(\eta_{1}),|\log \eta_{1}|,0.16\} = k\log 10  := B_{1}
\]
and 
\[
\max\{kh(\eta_{2}),|\log \eta_{2}|,0.16\} =  \log 3 := B_{2}.
\]
Therefore, by the estimate \eqref{eq 2.7} and the properties of logarithmic height, it follows that 
\begin{align*}
    h(\eta_{3}) & = h\left( \frac{9 (2\gamma - 2) g_{k}(\gamma)}{a}\right)\\ 
    & \leq h\left(\frac{9}{a}\right) + h(2\gamma - 2) + h(g_{k}(\gamma)) \\
    &< \log 9 + \log 4 + 5 \log (k) \\
    &< 10.2  \log k \ \text{for all} \ k \geq 2.
\end{align*}
Thus, we obtain
\[
\max\{kh(\eta_{3}),|\log \eta_{3}|,0.16\} = 10.2 k \log k := B_{3}.
\]
Since $2(l + m)-2.6 < n$. Therefore, we can take $D:= n$. Now by Theorem \ref{thm2}, we have
\begin{equation}\label{eq 3.17}
- \log |\Lambda_{1}| < 1.432 \times 10^{11} k^{2} (1+ \log k) (1 +\log n) (k \log 10)(\log 3)(10.2 k \log k) .    \end{equation}
The comparison of the lower bound \eqref{eq 3.17} and upper bound \eqref{eq 3.15} of $\Lambda_{1}$ gives us
\[
l \log 10 - \log 11.8 < 3.7 \times 10^{12} k^{4} \log k (1 + \log k)(1+ \log n).
\]
Using the facts $1+\log k < 2.5 \log k$ for all $k \geq 2$ and $1 + \log n < 1.8 \log n$ for $n \geq 4$, we can conclude that \begin{equation}\label{eq 3.18}
 l \log 10 < 1.68 \times 10^{13} k^{4} \log^{2} k \log n.
 \end{equation}
Now rearranging the equation \eqref{eq 1.2} as
\[
(2\gamma - 2) g_{k}(\gamma) \gamma^{n} - 10^{m} \left( \frac{a \cdot 10^{l} - a +b}{9} \right) = \frac{-b}{9} - \left( Q_{n}^{(k)} - (2\gamma - 2) g_{k}(\gamma) \gamma^{n} \right)
\]
and taking absolute values on both sides, we get
\[
\left| (2\gamma - 2) g_{k}(\gamma) \gamma^{n} - 10^{m} \left( \frac{a \cdot 10^{l} -a +b}{9} \right) \right| < 3.
\]
Dividing both sides of the above inequality by $(2\gamma - 2) g_{k}(\gamma) \gamma^{n}$ and using Lemma \ref{lem 2.1}, yields
\begin{equation}\label{eq 3.19}
    \left| 10^{m} \gamma^{-n} \left( \frac{a \cdot 10^{l} -a +b}{9 (2\gamma - 2) g_{k}(\gamma)} \right) - 1 \right| < \frac{3}{(2\gamma - 2) g_{k}(\gamma) \gamma^{n}} < \frac{5.5}{\gamma^{n}}.
\end{equation}
Let 
\begin{equation}\label{eq 3.20}
\Lambda_{2}:= 10^{m} \gamma^{-n} \left( \frac{a \cdot 10^{l} -a +b}{9(2\gamma - 2) g_{k}(\gamma)} \right) - 1.    
\end{equation}
Suppose that $\Lambda_{2}= 0$, then we get
\[
(2 \gamma -2)g_{k}(\gamma) \gamma^{n} = 10^{m} \cdot \left( \frac{a \cdot 10^{l} -a +b}{9} \right).
\]
This implies that $g_{k}(\gamma)$ is an algebraic integer, which is a contradiction. Thus, $\Lambda_{2} \neq 0$. We take $s := 3$,
\[
 \eta_{1}:= 10, \quad \eta_{2}:= \gamma, \quad \eta_{3}:= \frac{a \cdot 10^{l} -a +b}{9(2\gamma - 2) g_{k}(\gamma)} ,
\]
and
\[ 
a_{1}:= m, \quad a_{2}:= -n, \quad a_{3}:= 1.
\]
Note that  $\mathbb{L} := \mathbb{Q}(\gamma)$ contains $\eta_{1}, \eta_{2}, \eta_{3}$ and has $d_{\mathbb{L}}:= k$. Since $m < 2(l + m)-2.6 < n$, we deduce that $D:= \max\{|a_{1}|,|a_{2}|,|a_{3}|\}= n$. Since the logarithmic heights for $\eta_{1}$ and $\eta_{2}$ calculated as before are $h(\eta_{1}) = \log 10$ and $h(\eta_{2}) = (\log \gamma) / k < (\log 3)/k$. Therefore, we may take $ B_{1} := k \log 10$ and $B_{2} := \log 3$. By using \eqref{eq 2.7} and \eqref{eq 3.18}, we have
\begin{align*}
    h(\eta_{1}) &= h\left( \frac{a \cdot 10^{l} -a +b}{9(2\gamma - 2) g_{k}(\gamma)} \right) \\
    & \leq h(a)+ h(10^{l}) + h(a) + h(b) + \log 2 + h(9) + h(2\gamma - 2) + h(g_{k}(\gamma))  \\
    & < 4 \log 9 + l \log 10 + \log 2 + \log 4 + 5 \log k \\
    & < 4 \log 9 + 3 \log 2 + 1.68 \times 10^{13} k^{4} \log^{2} k \log n  + 5 \log k \\
    & < 1.69 \times 10^{13} k^{4} \log^{2} k \log n.
\end{align*}
As a result, we can take $B_{3} := 1.69 \times 10^{13} k^{5} \log^{2} k \log n.$ Consequently, by using Theorem \ref{thm2} with the facts $1+\log k < 2.5 \log k$ for $k \geq 2$ and $1 + \log n < 1.8 \log n$ for $n \geq 4$ to the inequality \eqref{eq 3.20}, we obtain that
\[
- \log |\Lambda_{2}| < 2.76 \times 10^{25} k^{8} \log^{3} k  \log^{2} n.
\]
Comparing the above inequality with \eqref{eq 3.19}, we get
\[
n < 4 \times 10^{25} k^{8} \log^{3} k  \log^{2} n,
\]
which leads to
\begin{equation}\label{eq 3.21}
\frac{n}{(\log n)^{2}} < 4 \times 10^{25} k^{8} \log^{3} k.    
\end{equation}
Thus, putting $S := 4 \times 10^{25} k^{8} \log^{3} k$ in \eqref{eq 3.21} and using Lemma \ref{lem 2.8}, the above inequality yields 
\begin{align*}
n & < 4(4 \times 10^{25} k^{8} \log^{3} k) (\log (4 \times 10^{25} k^{8} \log^{3} k))^{2} \\
& < (1.6 \times 10^{26} k^{8} \log^{3} k) (58.95 + 8 \log k + 3 \log (\log k))^{2} \\
& < (1.6 \times 10^{26} k^{8} \log^{3} k) (91.5 \log k) ^{2} \\
& < 1.34 \times 10^{30} k^{8} \log^{5} k,    
\end{align*}
where we have used the fact that $58.95 + 8 \log k + 3 \log (\log k) < 91.5 \log k $ for $k \geq 2$. \\
The result established in this subsection is summarized in the following lemma.
\begin{lemma}\label{lem3.1}
All the solutions of equation \eqref{eq 1.2} satisfy
\begin{equation}\label{eq 3.22}
 n < 1.34 \times 10^{30} k^{8} \log^{5} k.
\end{equation}
\end{lemma}
\subsection{An absolute upper bound on $n$}
For $k \geq 550$, the following inequalities hold
\[
 n < 1.34 \times 10^{30} k^{8} \log^{5} < \phi^{k/2}.
\]
Thus, from Lemma \ref{lem 2.4}, we have
\begin{equation}\label{eq 3.23}
Q_{n}^{(k)} = \frac{2\phi^{2n+1}}{\phi +2} (1 + \zeta ), \quad |\zeta| < \frac{41}{\phi^{k/2}}. \end{equation}
Since $k \geq 550$, therefore it is seen that $1 + \zeta \in \left( \frac{1}{2}, \frac{3}{2}
\right)$. Hence, by using \eqref{eq 3.12}, we get
\begin{equation}\label{eq 3.24}
 \frac{2\phi^{2n+1}}{\phi +2} > \frac{Q_{n}^{(k)}}{3/2} > \frac{2 \cdot 10^{l+m-1}}{3} = \frac{10^{l+m}}{15}.
\end{equation}
Inserting \eqref{eq 3.23} in \eqref{eq 1.2}, we obtain
\[
\frac{2\phi^{2n+1}}{\phi +2} (1 + \zeta ) = a \left ( \frac{10^{l}-1}{9} \right) \cdot 10^{m} + b \left ( \frac{10^{m}-1}{9} \right),
\]
which can be arranged by taking absolute values as 
\begin{align*}
\left|\frac{2\phi^{2n+1}}{\phi +2} -\frac{a}{9}10^{l+m}\right| & \leq  \frac{2\phi^{2n+1}}{\phi +2}|\zeta| + \left|\frac{10^{m}(b-a) - b}{9} \right| \\
& < \frac{82\phi^{2n+1}}{(\phi +2)\phi^{k/2}} + 1.1 \cdot 10^{m}.
\end{align*}
Dividing both sides of the above inequality by $\frac{2\phi^{2n+1}}{\phi +2}$ and using \eqref{eq 3.24}, we have
\begin{align}\label{eq 3.25}
\left|\frac{a(\phi +2)}{18}\phi^{-(2n+1)}10^{l+m} -1 \right| & < \frac{41}{\phi^{k/2}} + \frac{1.1 \cdot (\phi +2) \cdot 10^{m}}{2\phi^{2n+1}} \nonumber \\
& < \frac{41}{\phi^{k/2}} + \frac{1.1 \cdot 15 \cdot 10^{m}}{10^{l+m}} \nonumber \\
& < \frac{41}{\phi^{k/2}} + \frac{16.5}{10^{l}} \nonumber \\
& < 57.5 \cdot \max \left\{ \frac{1}{\phi^{k/2}}, \frac{1}{10^{l}} \right\}.  
\end{align}    
To apply Theorem \ref{thm2} in \eqref{eq 3.25}, set 
\[
\Lambda_{3} := \frac{a(\phi +2)}{18}\phi^{-(2n+1)}10^{l+m} - 1.
\]
If  $\Lambda_{3} = 0$, then 
\[
\frac{2\phi^{2n+1}}{\phi +2} = \frac{a}{9}10^{l+m},
\]
which is impossible because the right-hand side is rational whereas it can be seen that the left-hand side is irrational.
Therefore, $\Lambda_{3} \neq 0$. We take $s := 3$,
\[
 \eta_{1}:= \frac{a(\phi+2)}{18} , \quad  \eta_{2}:= \phi, \quad  \eta_{3}:= 10,
\]
and
\[ 
a_{1}:= 1, \quad a_{2}:= -(2n+1), \quad a_{3}:= l+m.
\]
Note that $\mathbb{L} := \mathbb{Q}(\sqrt{5})$ contains $\eta_{1}, \eta_{2}, \eta_{3}$ and has $d_{\mathbb{L}}:= 2$. Since $l+ m < 2n+1$, we deduce that $D:= \max\{|a_{1}|,|a_{2}|,|a_{3}|\}= 2n+1$. Moreover, since $h(\eta_{2}) = \frac{\log \phi}{2}$, $h(\eta_{3}) = \log 10$ and
\begin{align*}
 h(\eta_{1}) & \leq  h\left(\frac{a}{18}\right) + h(\phi)+ h(2) + \log 2 \\
 & \leq  h\left(\frac{a}{9}\right) + h(2) + h(\phi)+ h(2) + \log 2 \\
 & \leq  \log 9 +3 \log 2 + \frac{\log \phi}{2}  \\
 & \leq \log 72 + \frac{\log \phi}{2}.
 \end{align*}
Therefore, we may take
\[
B_{1} := \log \left( 72^{2} \cdot \phi  \right), \quad B_{2} := \log \phi \quad \text{and} \quad B_{3} := 2 \log 10.
\]
As before, applying Theorem \ref{thm2}, we have 
\begin{equation}\label{eq 3.26}
- \log |\Lambda_{3}| < 4.67 \times 10^{13} \log n,   
\end{equation}
where $1 + \log (2n+1) < 2.4 \log n $ holds for all $n \geq 4$. When we compare \eqref{eq 3.26} with \eqref{eq 3.25}, it yields 
\begin{equation}\label{eq 3.27}
\min \left\{ \frac{k}{2} \log \phi, l \log 10 \right\} < 4.67 \times 10^{13} \log n.
\end{equation}
Now, we distinguish two cases according to the above inequality \eqref{eq 3.27}.\\
\textbf{Case 1} Suppose that $\frac{k}{2} \log \phi = \min \left\{ \frac{k}{2} \log \phi, l \log 10 \right\}$. In this case, we have
\[
k < 1.95 \times 10^{14} \log n.
\]
Putting the above result into \eqref{eq 3.22}  and using the fact $k < n$, it yields
\begin{align*}
n & < 1.34 \times 10^{30} (1.95 \times 10^{14} \log n)^{8} (\log  n)^{5} \\
& < 2.81 \times 10^{144} (\log n)^{13}.
\end{align*}
Further using Lemma \ref{lem 2.8}, it leads to 
\begin{equation}\label{eq 3.28}
    n < 1.41 \times 10^{181}.
\end{equation}
\textbf{Case 2} Suppose that $l \log 10 = \min \left\{ \frac{k}{2} \log \phi, l \log 10 \right\}$, then we have 
\begin{equation}\label{eq 3.29}
   l \log 10 < 4.67 \times 10^{13} \log n.
\end{equation}
Using inequality \eqref{eq 3.23}, we can rearrange \eqref{eq 1.2} as
\[
\frac{2 \phi^{2n+1}}{\phi +2} - 10^{m} \left( \frac{a \cdot 10^{l}-a+b}{9} \right) = \frac{-b}{9} - \zeta \frac{2 \phi^{2n+1}}{\phi +2}. 
\]
Taking absolute values on both sides, we obtain
\[
\left| \frac{2 \phi^{2n+1}}{\phi +2} - 10^{m} \left( \frac{a \cdot 10^{l}-a+b}{9} \right) \right| =  \frac{b}{9} + |\zeta| \frac{2 \phi^{2n+1}}{\phi +2}, 
\]
where $|\zeta| < \frac{41}{\phi^{k/2}}$. Dividing $\frac{2 \phi^{2n+1}}{\phi +2}$ on both sides of the above equation, we get 
\begin{align}\label{eq 3.30}
\left|10^{m} \phi^{-(2n+1)} \frac{(a \cdot 10^{l}-a+b)(\phi +2)}{18}  - 1 \right| & <  \frac{b(\phi +2)}{18 \phi^{2n+1}} + |\zeta| \nonumber \\  
& < \frac{\left(\frac{\phi+2}{2\phi}\right)}{\phi^{2k}} + \frac{41}{\phi^{k/2}} \nonumber \\
& < \frac{1.12}{\phi^{2k}} + \frac{41}{\phi^{k/2}} < \frac{42.12}{\phi^{k/2}}.    
\end{align}
Now apply Theorem \ref{thm2} to \eqref{eq 3.30} with the data $s:= 3$ and 
\[
\Lambda_{4} := 10^{m} \phi^{-(2n+1)} \frac{(a \cdot 10^{l}-a+b)(\phi +2)}{18}  - 1,
\]
where 
\[
 \eta_{1}:= 10 , \quad  \eta_{2}:= \phi, \quad  \eta_{3}:= \frac{(a \cdot 10^{l}-a+b)(\phi +2)}{18},
\]
and
\[ 
a_{1}:= m, \quad a_{2}:= -(2n+1), \quad a_{3}:= 1.
\]
Here $\Lambda_{4} \neq 0$. If it were, then we would get a contradiction just as $\Lambda_{3}$. Thus $\Lambda_{4} \neq 0$. As previously calculated, we can take
 \[
 d_{\mathbb{L}}:= 2, \quad B_{1}:= 2 \log 10, \quad \quad B_{2}:=  \log \phi \quad  and \quad D:= 2n+1.
 \]
Moreover, by using \eqref{eq 3.29}, we have \begin{align*}
h(\eta_{3}) & = h \left(\frac{(a \cdot 10^{l}-a+b)(\phi +2)}{18} \right) \\
& \leq h(a \cdot 10^{l}-a+b) + h(\phi +2) + h(18) \\
& \leq h(a) + h(a-b) + \log 2 + l h(10)  + \log 2 + h(\phi) + \log 2 + h(18) \\ 
& \leq 3 \log 9 + 4 \log 2+ l \log 10 + \frac{\log \phi}{2} \\
& \leq 3 \log 9 + 4 \log 2 +  \frac{\log \phi}{2} + 4.67 \times 10^{13} \log n  \\
& \leq \log 27 + \log 16 + \frac{\log \phi}{2} + 4.67 \times 10^{13} \log n \\
& < 4.68 \times 10^{13} \log n. 
\end{align*}
Hence, we can take $B_{3}:= 9.36 \times 10^{13} \log n.$ According to Theorem \ref{thm2} and \eqref{eq 3.30}, we obtain 
\[
\frac{k}{2} \log \phi - \log 42.12 < 4.83 \times 10^{26} (\log n)^{2}
\]
and hence
\begin{equation}\label{eq 3.31}
k < 2.1 \times 10^{27} (\log n)^{2}.   
\end{equation}
Inserting \eqref{eq 3.31} into \eqref{eq 3.22} with the fact $k < n$, we obtain
\begin{align*}
n & < 1.34 \times 10^{30} (2.1 \times 10^{27} \log^{2} n)^{8} (\log n)^{5} \\
& < 5.1 \times 10^{248} (\log n)^{21},
\end{align*}
which implies that
\begin{equation}\label{eq 3.32}
    n < 8.82 \times 10^{312}.
\end{equation}
In both cases, we find that \eqref{eq 3.32} is always true. This bound on $n$ is valid for $k > 550$. If $k \leq 550$, then we have from \eqref{eq 3.22} that
\begin{align*}
n & < 1.34 \times 10^{30} (550)^{8} [\log (550)]^{5} \\
& < 1.13 \times 10^{56}.
\end{align*}
As a result, \eqref{eq 3.32} holds for all $k$.
The result established in this subsection is summarized in the following lemma.
\begin{lemma}\label{lem 3.2}
Let $n > 250$. If $n$ is a solution of the Diophantine equation \eqref{eq 1.2},
then
\[
n < 8.82 \times 10^{312}.
\]
\end{lemma}
\subsection{Reducing the bound on $n$}
The upper bound on $n$ given in Lemma \ref{lem 3.2} is too large for computation, therefore we will now apply Lemma \ref{lem 2.6} to reduce it. Let
\[
\Gamma_{3} := (l+m) \log 10 - (2n+1) \log \phi + \log \left( \frac{a(\phi +2)}{18} \right).
\]
From inequality \eqref{eq 3.25}, we have that
\[
|\Lambda_{3}| = |e^{\Gamma_{3}} - 1| < 57.5 \cdot \max \left\{ \frac{1}{\phi^{k/2}}, \frac{1}{10^{l}} \right\} < \frac{57.5}{e^{w}},
\]
where $w = \min \left\{ \frac{k}{2} \log \phi, l \log 10 \right\}$. Assume that $l \geq 2$. Then $|\Lambda_{3}| < \frac{57.5}{100} < 0.58$. Choosing $a: = 0.58$, we obtain the inequality 
\[
|\Gamma_{3}| = |\log (\Lambda_{3} + 1)| < \frac{- \log(1 - 0.58)}{0.58} \cdot \frac{57.5}{e^{w}} < \frac{86.1}{e^{w}}.
\]
Dividing both sides of the above inequality by $\log \phi$, we get 
\begin{equation}\label{eq 3.33}
    0 < |(l+m) \log \widehat{\tau}- (2n +1) +\widehat{\mu}| < A \cdot C^{-w},
\end{equation}
where 
\[
\widehat{\tau} := \frac{\log 10}{\log \phi}, \quad \widehat{\mu} := \frac{ \log \left( \frac{a(\phi +2)}{18} \right)}{\log \phi}, \quad A:= 178.92, \quad C:= e.
\]
Here, $M := 8.82 \times 10^{312}$ is an upper bound on $l+m$ since $l+m < n < 8.82 \times 10^{312}$ by \eqref{eq 3.13} and Lemma \ref{lem 3.2}. We found that $q_{620}$, the denominator of the $620$th convergent of $\widehat{\tau}$ exceeds $6M$. Hence, applying Lemma \ref{lem 2.6} to the inequality \eqref{eq 3.33}, a quick computation with \textit{Mathematica} gives us that the value
\[
\frac{\log ( A q_{620}/\epsilon)}{\log C}
\]
is less than $744.38$. So, if the inequality $\eqref{eq 3.33}$ has a solution, then 
\[
w < \frac{\log ( A q_{620}/\epsilon)}{\log C} < 744.38.
\]
Assume that $k \geq 3200$. In this case, $\frac{k}{2} \log \phi > 744$. Thus $w = l \log 10 \leq 744$, which implies that $l \leq 323$.

Now, applying Lemma \ref{lem 2.6} to \eqref{eq 3.30}. Assume that  
\[
\Gamma_{4} := m \log 10 - (2n+1) \log \phi + \log \left( \frac{(a \cdot 10^{l}-a+b)(\phi +2)}{18} \right).
\]
Then $\Lambda_{4} = e^{\Gamma_{4}} -1$ and from \eqref{eq 3.30}, we have 
\[
|\Lambda_{4}| = |e^{\Gamma_{4}} -1| < \frac{42.12}{\phi^{k/2}} < \frac{1}{10}.
\]
Choosing $a: = 0.1$ in Lemma \ref{lem 2.7}, we obtain the inequality 
\[
|\Gamma_{4}| = |\log (\Lambda_{4} + 1)| < \frac{- \log(1 - 0.1)}{0.1} \cdot \frac{42.12}{\phi^{k/2}} < \frac{44.38}{\phi^{k/2}}.
\]
Dividing this inequality by $\log \phi$, we get 
\begin{equation}\label{eq 3.34}
    0 < |m \widehat{\tau}- (2n +1) +\widehat{\mu}| < A \cdot C^{-(k/2)},
\end{equation}
where 
\[
\widehat{\tau} := \frac{\log 10}{\log \phi}, \quad \widehat{\mu} := \frac{ \log \left( \frac{(a \cdot 10^{l}-a+b)(\phi +2)}{18} \right)}{\log \phi}, \quad A:= 92.22, \quad C:= \phi.
\]
Taking 
\[
M := 8.82 \times 10^{312},
\]
which is an upper bound on $m$ since $m < n < 8.82 \times 10^{312}$ by \eqref{eq 3.13} and Lemma \ref{lem 3.2}. We found that $q_{630}$, the denominator of the $630$th convergent of $\widehat{\tau}$ exceeds $6M$. Hence, applying Lemma \ref{lem 2.6} to the inequality \eqref{eq 3.34}, a quick computation with \textit{Mathematica} gives us that
\[
\frac{k}{2} < \frac{\log ( A q_{630}/\epsilon)}{\log C} < 1584,
\]
that is $k \leq 3166$. This contradicts our assumption that $k \geq 3200$. Thus we proved that if $k \geq 550$, then $k < 3200$.

Now, we assume that $k > 550$. Since $k < 3200$, it follows from \eqref{eq 3.21} that
\begin{align*}
     n & < 1.34 \cdot 10^{30} \cdot 3200^{8} \cdot (\log 3200)^{5} \\
     & < 5.1 \cdot 10^{62}.
\end{align*}
Now, if we again apply Lemma \ref{lem 2.6} to \eqref{eq 3.33} with $M := 5.1 \cdot 10^{62}$, which is an upper bound on $l + m$ as $l + m < n < 5.1 \cdot 10^{62}$, we get
\[
w < \frac{\log ( A q_{135}/\epsilon)}{\log C} < 154.342.
\]
Since $\frac{k}{2} \log \phi > 132$ for $k > 550$, it is clear that $w = l \log 10$, which implies
that $l \leq 66$. Similarly, applying Lemma \ref{lem 2.6} to \eqref{eq 3.34} and using this value of $M$. In this case, we have $k \leq 550$. This contradicts our assumption that $k > 550$. Thus, $k \leq 550$. Substituting this upper bound on $k$ into \eqref{eq 3.21}, we obtain
\begin{align*}
n & < 1.34 \cdot  10^{30} \cdot 550^{8} \cdot (\log 550)^{5} \\
& < 1.13 \cdot 10^{56}.
\end{align*}
Now, let us apply Lemma \ref{lem 2.6} to \eqref{eq 3.15}. For this, let
\[
\Gamma_{1} := (l+m) \log 10 - n \log \gamma + \log \left( \frac{ 9 (2 \gamma -2) g_{k}(\gamma)}{a} \right).
\]
So, $\Lambda_{1} = e^{\Gamma_{1}} - 1$. Assuming that $m \geq 2$, it is seen that from  \eqref{eq 3.15} that  $|\Lambda_{1}| < \frac{11.8}{100} < 0.12$. Choosing $a: = 0.12$ in Lemma \ref{lem 2.7}, we obtain
\[
|\Gamma_{1}| = |\log (\Lambda_{1} + 1)| < \frac{- \log(1 - 0.12)}{0.12} \cdot \frac{11.8}{10^{l}} < \frac{12.58}{10^{l}}.
\]
Dividing this inequality by $\log \gamma$, we get 
\begin{equation}\label{eq 3.35}
    0 < |(l+m) \widehat{\tau}- (2n +1) +\widehat{\mu}| < A \cdot C^{-l},
\end{equation}
where 
\[
\widehat{\tau} := \frac{\log 10}{\log \gamma}, \quad \widehat{\mu} := \frac{ \log \left( \frac{ 9 (2 \gamma -2) g_{k}(\gamma)}{a} \right)}{\log \gamma}, \quad A:= 18.15, \quad C:= 10.
\]
Put 
\[
M := 1.13 \cdot 10^{56},
\]
in Lemma \ref{lem 2.6}. We found for all $k \in [2, 550]$ that $q_{135}$, the denominator of the $135$th convergent of $\gamma$ exceeds $6M$. A quick computation with \textit{Mathematica}
gives us that if the inequality \eqref{eq 3.15} has a solution, then
\[
l < \frac{\log ( A q_{135}/\epsilon)}{\log C} < 89.14.
\]
That is, $l \leq 89$. 

\noindent
Finally, assume that
\[
\Gamma_{2}:= m \log 10 -n \log \gamma + \log  \left( \frac{a \cdot 10^{l} -a +b}{9(2\gamma - 2) g_{k}(\gamma)} \right).
\]
Then $\Lambda_{2} = e^{\Gamma_{2}} - 1$ and we can see from \eqref{eq 3.25} that $|\Lambda_{2}| < \frac{5.5}{\gamma^{n}} < 0.1$ as $n > 250$. Taking $a := 0.1$ in Lemma \ref{lem 2.7}, we obtain
\[
0 < \left|m \log 10 - n \log \gamma + \log \left( \frac{a \cdot 10^{l} -a +b}{9(2\gamma - 2) g_{k}(\gamma)} \right) \right|< \frac{- \log(1 - 0.1)}{0.1} \cdot \frac{5.5}{\gamma^{n}} < \frac{5.8}{\gamma^{n}}.
\] 
Dividing this inequality by $\log \gamma$ gives us the inequality
\begin{equation}\label{eq 3.36}
    0 < \left| m \frac{\log 10}{\log \gamma} - n  + \frac{\log \left( \frac{a \cdot 10^{l} -a +b}{9(2\gamma - 2) g_{k}(\gamma)} \right)}{\log \gamma} \right| < \frac{8.37}{\gamma^{n}}.
\end{equation}
Let
\[
\widehat{\tau} := \frac{\log 10}{\log \gamma}, \quad M :=  1.13 \cdot 10^{56}, \quad A:= 8.37, \quad C:= \gamma, \quad w:= n
\]
and
\[
\widehat{\mu} := \frac{ \log \left( \frac{ 9 (2 \gamma -2) g_{k}(\gamma)}{a} \right)}{\log \gamma}
\]
in Lemma \ref{lem 2.6}. If the inequality \eqref{eq 3.36} has a solution, then 
\[
w:= n < \frac{\log ( A q_{150}/\epsilon)}{\log C} \leq 219.568.
\]
But, this contradicts the fact that $n > 250$. Thus, the proof is completed.

\vspace{05mm} \noindent \footnotesize
\begin{minipage}[b]{90cm}
\large{School of Applied Sciences, \\ 
KIIT University, Bhubaneswar, \\ 
Bhubaneswar 751024, Odisha, India. \\
Email: bptbibhu@gmail.com}
\end{minipage} 

\vspace{05mm} \noindent \footnotesize
\begin{minipage}[b]{90cm}
\large{School of Applied Sciences, \\ 
KIIT University, Bhubaneswar, \\ 
Bhubaneswar 751024, Odisha, India. \\
Email: bijan.patelfma@kiit.ac.in}
\end{minipage}

\end{document}